\documentclass{article}
\usepackage[linesnumbered,ruled,vlined]{algorithm2e}

\usepackage{amsmath,amssymb,amsthm}
\usepackage[hidelinks]{hyperref}
\usepackage{yfonts}
\usepackage{enumerate}
\usepackage{color}
\usepackage{mathtools}
\usepackage{thmtools,thm-restate}
\usepackage{tikz}
\usepackage[justification=centering]{caption}
\usepackage{ wasysym } 
\usepackage{verbatim}
\usepackage{bbm}

\usepackage{geometry}

\usepackage{graphicx} 

\newtheorem{theorem}{Theorem}[section]
\newtheorem{lemma}[theorem]{Lemma}
\newtheorem{proposition}[theorem]{Proposition}

\newtheorem{fact}[theorem]{Fact}
\newtheorem{remark}[theorem]{Remark}

\title{On the approximability of the burning number}
\author{Anders Martinsson\thanks{Department of Computer Science, ETH Z\"urich, Switzerland\newline anders.martinsson@inf.ethz.ch}}
\date{\today}

\begin{document}

\maketitle

\begin{abstract} The burning number of a graph $G$ is the smallest number $b$ such that the vertices of $G$ can be covered by balls of radii $0, 1, \dots, b-1$. As computing the burning number of a graph is known to be NP-hard, even on trees, it is natural to consider polynomial time approximation algorithms for the quantity. The best known approximation factor in the literature is $3$ for general graphs and $2$ for trees. In this note we give a $2/(1-e^{-2})+\varepsilon=2.313\dots$-approximation algorithm for the burning number of general graphs, and a PTAS for the burning number of trees and forests. Moreover, we show that computing a $(\frac53-\varepsilon)$-approximation of the burning number of a general graph $G$ is NP-hard.
\end{abstract}

\section{Introduction}

The burning number of a graph $G$, denoted $b(G)$ is the smallest number $b$ such that the vertices of $G$ can be covered by balls of radii $0, 1, \dots, b-1$. We can think of this as a process, or one-player game, of burning the graph $G$. Initially no vertices in $G$ are burning. At each time step $t=1, 2, \dots$, we get to pick a vertex $v_t$ and set it on fire. A vertex $v$ is burning at time $t$ if either $v=v_t$, or $v$ or one of its neighbors was burning at time $t-1$. Thus the burning number $b(G)$ denotes the earliest possible $t$ when all vertices of $G$ are on fire.

This quantity first appeared in print in a paper by Alon \cite{Alon92} who determined the burning number of the hypercube. It was later independently proposed by Bonato, Janssen and Roshanbin \cite{bonato15}. Since its reintroduction, the problem has gained quite some attention, focusing mainly on determining the burning number of particular classes of graphs, or the complexity of graph burning. In particular, Bonato et al \cite{bonato15} showed that $b(P_n)=\lceil \sqrt{n}\rceil$, and consequently conjectured that $b(G)\leq \lceil \sqrt{n}\rceil$ for all connected graphs on $n$ vertices. Following a sequence of papers proving $b(G)\leq (C+o(1))\sqrt{n}$ for successively smaller constants $C$ \cite{bonato15, burn1, burn2, pyrotechnics}, it was recently shown by Norin and Turcotte \cite{NT22} that $b(G)\leq (1+o(1))\sqrt{n}$, thus resolving the conjecture asymptotically. We refer to \cite{Bonatosurvey} for a recent survey of futher research on the topic.

In this note, we are interested in the computability of $b(G)$. It is known that computing the burning number of a graph is NP-complete. In fact, the problem is still hard if one restricts attention to very simple graph classes such as caterpillars of maximum degree $3$ \cite{Bessy17, Hiller21, Liu20, Gupta20+}.

As computing $b(G)$ exactly is hard, it is natural to ask for efficient approximation algorithms. As was observed by Bessy, Bonato, Janssen and Rautenbach  \cite{Bessy17}, a greedy assignment scheme gives a $3$-approximation of $b(G)$. We here give a slight variation of the algorithm.

\begin{algorithm}[H]\label{algo:1}
\SetAlgoLined
\KwData{$G=(V, E)$}
$r\leftarrow 0$\;
\While{$\exists$ an uncovered vertex $v$}{
    Place a ball with radius $r$ centered at $v$\;
    $r\leftarrow r+1$\;
}
\Return{r}\;
\caption{Greedy burning}
\end{algorithm}
The outputted value $r$ of the above algorithm is a $3$-approximation for $b(G)$: On the one hand, the algorithm clearly produces a feasible solution, so $b(G)\leq r$. On the other hand, as the centers of the $\lfloor r/3 \rfloor$ largest balls are at pairwise distance at least $2r/3$, no two of these vertices can be covered with a ball of radius $<r/3$, so $b(G)\geq r/3.$

For special cases of graphs, efficient algorithms with better approximation factors are known. In particular, Bonato and Lidbetter  \cite{BonatoLibetter19} proposed a $\frac32$-approximation algorithm when the graph is a path-forests (that is, a graph where each connected compontent is a path). Bonato and Kamali \cite{BonatoKamali19} gave $2$-approximation algorithm for the burning number of trees, and moreover gave an FPTAS for path-forests.

Given their results on path-forests, Bonato and Kamali posed the question of whether there is an FPTAS for the burning number of general graphs. This turns out to be false, assuming P$\neq$NP. Mondal, Rajasingh, Parthiban and Rajasingth \cite{Mondal22}, showed that that computing the burning number of a graph is APX-hard. In particular, there exists an $\varepsilon>0$ such that computing a $1+\varepsilon$-approximation of the burning number is NP-hard.

In this note, we aim to improve the best known approximation factors for the burning number. Our first result, improve the best known approximation factor of $b(G)$ that can be efficiently computed from $3$ to $2.313\dots$.

\begin{theorem}\label{thm:algo} There exists a polynomial time randomized algorithm that, given a graph $G$, outputs a number $r\geq b(G)$, such that $$r\leq \left(\frac{2}{1-e^{-2}}+e_{b(G)}\right) b(G)$$ with probability $1-O(1/poly(n))$ where $e_b$ denotes a function tending to $0$ as $b\rightarrow\infty$.
\end{theorem}

This will be shown in Section \ref{sec:algo}. For any $\varepsilon>0$, this computes a $\left(2/(1-e^{-2})+\varepsilon\right)$-approximation for $b(G)$, provided $b(G)$ is sufficiently large depending on $\varepsilon$. We remark that this can be extended to a polynomial time algorithm that computes a $\left(2/(1-e^{-2})+\varepsilon\right)$-approximation for $b(G)$ for all $G$ by using brute force search to handle the cases where $b(G)$ is small.

It is natural to ask for the best approximation factor $\alpha$ of the burning number that can be computed in polynomial time for a general graph. As a second result, we note that $\alpha\geq \frac53$.
\begin{proposition}\label{prop:inapprox} For any $\varepsilon>0$, it is NP-hard to approximate the burning number within a factor of $\frac53-\varepsilon$.
\end{proposition}
\begin{proof} We here make use of the fact that it is NP-hard to approximate the domination number of a graph $G=(V, E)$ within a factor of $O(\log n)$ \cite{RazSafra97}. In particular, given a graph $G$ and an integer $d$, it is NP-hard to distinguish between the cases where the domination number of $G$ is less than $\varepsilon d$ or at least $10d$. Without loss of generality we may assume that $1\leq d \leq |V|$.

We can distinguish these two cases as follows. Let $G$ and $d$ be given, and assume that $G$ either has domination number less than $\varepsilon d$ or at least $10d$. Construct the graph $G'$ as follows. Let $V'$ be a disjoint copy of the set $V$. Connect any two vertices $v, w\in V$ that are adjacent in $G$ by a path of length $2d$. Moreover connect any vertex $v\in V$ to its copy $v'\in V'$ by a path of length $d$. We claim that if $G$ has domination number less than $\varepsilon d$, then $b(G')< (3+\varepsilon)d$, and if $G$ has domination number at least $10d$, then $b(G')\geq 5d$. As any $(\frac53-\varepsilon)$-approximation algorithm for $b(G)$ can distinguish these two cases for $b(G')$, it follows that approximating the burning number of a graph up to this factor is NP-hard, as desired.

It remains to show this relation between the domination number of $G$ and the burning number of $G'$. First suppose that $G$ has a dominating set $D$ of size $<\varepsilon d$. We can burn $G'$ by first setting fire to all vertices in $D$ and then waiting for $3d$ time steps: As $D$ is a dominating set in $G$, it follows that the fire reaches all vertices in $V$ after at most $2d$ time steps. By construction of $G'$, all vertices are at distance at most $d$ from a vertex in $V$, hence waiting $d$ additional steps ensures that all vertices are on fire. Hence $b(G')<(3+\varepsilon)d$.

Conversely, suppose that $b(G') < 5d$, and consider a corresponding covering of $G'$ with $b(G')$ balls. Let $D\subseteq V$ be the set of vertices constructed as follows. For each ball whose center is on the path strictly between two vertices $v, w\in V$, we add both $v$ and $w$ to $D$. For any ball whose center is either a vertex $v\in V$, its copy $v'\in V'$, or a vertex on the path between these vertices, we add $v$ to $D$. Then $|D|\leq 2b(G')<10d$. To see that $D$ is a dominating set in $G$, note that any vertex $v'\in V'$ is at distance strictly less than $5d$ from a vertex in $D$ in $G'$, hence any vertex $v\in V$ is at distance strictly less than $4d$ from a vertex in $D$, meaning it is at distance strictly less than $2$ from a vertex in $D$ in $G$. Hence, $D$ is a dominating set of size less than $10d$. By contraposition, we get that if $G$ has domination number at least $10d$, then $b(G')\geq 5d$, proving the second direction of the desired relation.
\end{proof}

Finally, we give a PTAS for the burning number of trees and forests.

\begin{theorem} For any $\varepsilon>0$, there exists an algorithm that computes the burning number of a forest on $n$ vertices within a factor of $1+\varepsilon$ in $n^{O(1/\varepsilon)}$ time.
\end{theorem}

In fact, the solution is quite simple. Given a rough estimate $r$ for the size of $b(T)$ and a small constant $\varepsilon>0$, we consider the problem of covering $T$ with balls of radii multiple of $\varepsilon r$. On the one hand, this only shifts the optimal value by an additive error of $\varepsilon r$. On the other hand, the rounding reduces the search space enough that the problem can be solved in polynomial time using dynamic programming. This will be shown in Section \ref{sec:dp}.

\begin{remark} Shortly after this note was uploaded to arxiv, a paper appeared in the proceedings of APPROX 2023 by Lieskovsk\'y, Sgall and Feldmann \cite{LSF23} which independently prove Theorem \ref{thm:algo} as well as proving that approximating $b(G)$ within a factor of $\frac43$ is NP-hard.
\end{remark}

\section{Greedy random burning}\label{sec:algo}

The core of the proof of Theorem \ref{thm:algo} is the following modification of Algorithm \ref{algo:1}. Given a graph $G$ and a number $m$, we consider the following algorithm.

\begin{algorithm}[H]\label{algo:2}
\SetAlgoLined
\KwData{$G=(V, E)$, parameter $m$}
$t\leftarrow 0$\;
\While{$\exists$ an uncovered vertex $v$}{
    Place a ball with radius $r_t\sim U[0, m]$ centered at $v$\;
    $t\leftarrow t+1$
}
\caption{Greedy random burning}
\end{algorithm}

\begin{lemma}\label{lem:linexp} If $m\geq (2/(1-e^{-2})) b(G)$, then Algorithm \ref{algo:2} places, in expectation, at most $m$ balls before terminating. 
\end{lemma}
\begin{proof} Let $B_0, B_1, \dots, B_{b(G)-1}$ be any covering of $G$ with balls of radii $0, 1, \dots, b(G)-1$. Consider the number of times the algorithm tries to place a ball centered in $B_i$ for some fixed $i$. If at any such time, the algorithm chooses a radius $r_t\geq 2\cdot i$, then the entire set $B_i$ is covered and will not be chosen again. The probability that this happens is $1-2i/m$. Thus, in expectation, the algorithm places at most $1/(1-2i/m)$ balls centered at a vertex in $B_i$. By linearity of expectation, the algorithm places at most
$$\sum_{i=0}^{b(G)-1} \frac1{1-2i/m} \leq \int_0^{b(G)} \frac{1}{1-2x/m}\,dx = \frac{m}{2} \ln\left(\frac{1}{1-2b(G)/m} \right) $$
balls in total. It is readily seen that this $\leq m$ when $m\geq (2/(1-e^{-2})) b(G)$.
\end{proof}

We say that a sequence $(a_1, \dots, a_k)$ is dominated by a sequence $(b_1, \dots, b_\ell)$ if $k\leq \ell$ and, for each element $a_i$, we can uniquely assign an element $b_{j_i}$ such that $a_i \leq b_{j_i}$. The following is a standard consequence of Chernoff bounds.

\begin{fact}\label{fact:chernoff} Consider a sequence $(r_1, \dots, r_{m})$ of $m$ independent numbers distributed uniformly between $0$ and $m$. With probability at least $1-1/2(m+1)$, this is dominated by the sequence $(0, 1, 2, \dots, m+O(\sqrt{m\log m})-1)).$
\end{fact}

\begin{proof}[Proof of Theorem \ref{thm:algo}.] Let $G$ be given, and suppose we can guess a value $m$ which is between $(2/(1-e^{-2}))b(G)$ and $(2/(1-e^{-2}))b(G)+O(\sqrt{m \log m})$.

Run Algorithm \ref{algo:2} on $G$ with parameter $m$. Compute the smallest $r$ such that the sequence of radii is dominated by $(0, 1, 2, \dots, r-1)$. This always gives us an upper bound for $b(G)$. Moreover, it follows, by Lemma \ref{lem:linexp} and Markov's inequality, that Algorithm \ref{algo:2} uses at most $m$ balls with probability at most $1-m/(m+1)=1/(m+1)$. By Fact \ref{fact:chernoff}, the corresponding sequence is dominated by $(0, 1, \dots m+O(\sqrt{m\log m})-1)$ hence proving $b(G)\leq m+O(\sqrt{m\log m})$ with probability at least $1/(m+1)-1/2(m+1)\geq 1/2(m+1).$ Hence repeating this procedure $\Theta(m \log n)$ times, the algorithm will successfully prove that $b(G)\leq (1+\varepsilon)m$ at least once with probability $1-O(1/poly(n)).$

To conclude the proof, note that the procedure above never produces an incorrect upper bound on $b(G)$. Hence, we can simply try all values of possible values of $m$ one at a time.
\end{proof}

\begin{remark} We end this section with two short comments. First, it is possible to obtain the same approximation factor using a deterministic algorithm. Second, we believe that it should be possible to improve the approximation factor of the algorithm somewhat by biasing the probability distribution in Algorithm \ref{algo:2} towards picking small radii slightly more often, and compensating for this by removing any ball placed by the algorithm which is completely contained in a larger ball placed after it. For the sake of brevity, we will not attempt either of these modifications here.
\end{remark}

\section{Burning forests with dynamic programming}\label{sec:dp}

In this section, we will show how to use dynamic programming to approximate the burning number of forests within a factor of $1+O(\varepsilon)$ in polynomial time. 

We call a multiset of non-negative integers $\{r_1, \dots, r_k\}$ a cover of $G$ if it is possible to partition the vertices of $G$ into connected components $A_1, \dots, A_k$ such that the radius of each component $C_i$ is at most $r_i$.

\begin{lemma}\label{lem:dp} Let $a$ be a positive integer.\begin{enumerate}
    \item If $b(G)\leq b$, then there exists a subset of the multiset $\{ a\lceil i/a\rceil :0 \leq i < b \}$ that is a cover of $G$.
    \item If a subset of the multiset $\{ a\lceil i/a\rceil :0 \leq i < b \}$ is a cover of $G$, then $b(G)\leq b+a-1$.
\end{enumerate}
\end{lemma}
\begin{proof}
    This follows directly from noting that $i \leq a\lceil i/a \rceil \leq i+a-1$.
\end{proof}

It follows from this that, in order to determine $b(G)$ up to an additive error of $a$, it suffices to determine the set of covers of $G$ consisting of multiples of $a$. Moreover, if we know an upper bound on $b(G)\leq m$, e.g. by using Algorithm \ref{algo:1}, it suffices to consider multiples of $a$ up to most $m$.

Crucially this means that if we wish to determine $b(G)$ up to a multiplicative error of $1+\varepsilon$, then we only need to consider covers consisting of $O(1/\varepsilon)$ different radii. So any cover can be represented as a $O(1/\varepsilon)$-dimensional vector of non-negative integers whose sum is, say, at most $|V(G)|$. In particular, this has polynomial size in $|V(G)|$, which is small enough that we can use dynamic programming in polynomial time.

Let us now describe how to use dynamic programming to determine the set of covers of a tree $G$. In the case where $G$ is a forest, we can simply apply this to each component of $G$ and combine the result.

Pick a vertex $r$ which we call the root of $G$. Any vertex of degree $1$ which is not $r$ is called a leaf. For any vertex $v$ we call its descendants, $D(v)$ the subgraph of $G$ formed by the set of all vertices whose path to $r$ goes through $v$, including $v$ itself.

For each vertex $v$ we define $\mathcal{C}_0(v)$ to be the set of covers of $D(v)$. For any $k\geq 1$, we similarly put $\mathcal{C}_k(v)$ to be the set of covers of any subgraph of the form $D(v)\setminus A$ where $A$ is either empty, or $A$ is a connected subgraph $v\in A\subseteq D(v)$ containing vertices at distance $<k$ from $v$. In other words, $\mathcal{C}_k(v)$ contains all the ways to cover all vertices of $D(v)$, except possibly for a not too large connected vertex set subset around $v$.

To initialize the recursion, we have for any leaf $v$ of $G$ that $\mathcal{C}_0(v)=\{(1, 0, 0, \dots), (0, 1, 0, \dots) \dots \}$, and $\mathcal{C}_k(v)=\{0, (1, 0, 0, \dots), (0, 1, 0, \dots) \dots \}$ for any $k\geq 1$.

Assuming the function $\mathcal{C}_k(w)$ has been determined for all children $w$ of a vertex $v$, we can compute $\mathcal{C}_0(v)$ as follows. First, in any cover of the descendants of $v$, $v$ itself must be contained in some connected component $A$ of radius at most $r$ for some multiple $r$ of $a$. This component must be centered at some vertex $v_0\in D(v)$, and have a radius $r$, which is a multiple of $a$. Let $v_1, \dots, v_\ell$ denote the vertices in $D(v)$ adjacent to the path from $v_0$ to $v$. For each of these, the intersection of $A\cap D(v_i)$ is either empty, or is a connected component containing vertices in $D(v_i)$ at distance at most $r-d(v_0, v_i)$ from $v_i$, including $v_i$ itself. Hence, we can write the corresponding set of covers associated with the choice of $v_0$ and $r$ as
$$\mathcal{C}_0(v, v_0, r) := e_r + \mathcal{C}_{r-d(v_0, v_1)-1}(v_1) + \dots + \mathcal{C}_{r-d(v_0, v_\ell)-1}(v_\ell),$$
where $e_r$ denotes the unit vector representing one element of size $r$, and where we interpret the sum of sets as the set of all sums that can be formed from elements in the corresponding sets. Finally this gives us $$\mathcal{C}_0(v) = \bigcup_{v_0\in D(v)} \bigcup_{r\geq d(v, v_0)} \mathcal{C}_0(v, v_0, r).$$

Given this, in order to compute $\mathcal{C}_k(v)$ for $k\geq 1$, let $w_1, \dots w_\ell$ denote the children of $v$. The set of covers where $v$ is not in a separate component is given by $\mathcal{C}_0(v),$ which we have already determined. Moreover, the set of covers where a connected component around $v$ is excluded we can express simply as
$$ \mathcal{C}_{k-1}(w_1) + \mathcal{C}_{k-1}(w_2) + \dots + \mathcal{C}_{k-1}(w_\ell).  $$

Having computed these sets for all vertices $v\in G$ and all $0\leq k \leq m$, we can simply read off the set of covers of $G$ from the set $\mathcal{C}_0(r)$. (If $G$ is a forest, similarly, the set of covers is the sumset of $\mathcal{C}_0(r_i)$ for the roots $r_1, r_2, \dots$ of each connected component.) We simply compare this to Lemma \ref{lem:dp} to find the smallest value $b$ for which both conclusions hold. Then we know $b \leq b(G) \leq b+a-1$, which gives us the desired approximation of $b(G)$.

In order to run this in polynomial time, we need efficient implementations of unions $\mathcal{A}\cup \mathcal{B}$ and sumsets $\mathcal{A}+\mathcal{B}$ of sets of covers of subgraphs of $G$. (Computing $\ell$-fold sums can be performed by calling the sumset function $\ell$ times.) As the set of covers we consider are represented as $O(1/\varepsilon)$-dimensional vectors of non-negative integers whose sum never exceeds $|V(G)|$, the sizes of the sets never exceed $|V(G)|^{O(1/\varepsilon)}$. Hence, any naïve implementation of these operations run in polynomial time.

\bibliographystyle{abbrv}
\bibliography{references}

\end{document}